\documentclass[letterpaper, 10 pt, conference]{ieeeconf}
\IEEEoverridecommandlockouts 

\newcounter{hugeEquationNumber}
\usepackage{amsmath}
\usepackage{graphicx}
\usepackage{amssymb}
\usepackage{amsmath}
\usepackage{epstopdf}
\usepackage{flushend,calc}

\usepackage{xcolor,calc}

\usepackage{parskip}
\makeatletter
\def\thm@space@setup{%
  \thm@preskip=\parskip \thm@postskip=0pt
}\makeatother

\newcommand{\norm}[1]{\lVert#1\rVert}
\renewcommand{\t}{^{T}}
\newcommand{\eqdef}{:=}
\newcommand{\mcf}{\mathcal}

\newcommand{\Hinf}{{\mathcal{H}_\infty}}
\newcommand{\inv}{^{-1}}

\newcommand{\Uset}{\mathbb{U}}
\newcommand{\Wset}{\mathbb{W}}
\newcommand{\Wsetinf}{\mathcal{W}}

\newtheorem{claim}{Claim}
\newtheorem{thm}{Theorem}

\newtheorem{lemma}{Lemma}
\newtheorem{remark}{Remark}

\title{Performance Bounds for Constrained Linear Min-Max Control}
\author{Tyler H. Summers$^*$ and Paul J. Goulart$^*$%
\thanks{T. H. Summers is supported by an ETH Zurich Postdoctoral Fellowship.}%
\thanks{$^*$Automatic Control Laboratory, Department of Information Technology and Electrical Engineering, ETH Zurich, 8092 Zurich, Switzerland. {\tt\small tsummers|pgoulart@control.ee.ethz.ch}}%
}
\date{\today} 

\begin{document}
\maketitle

\begin{abstract}
This paper proposes a method to compute lower performance bounds for discrete-time infinite-horizon min-max control problems with input constraints and bounded disturbances. Such bounds can be used as a performance metric for control policies synthesized via suboptimal design techniques. Our approach is motivated by recent work on performance bounds for stochastic constrained optimal control problems using relaxations of the Bellman equation. The central idea of the paper is to find an unconstrained min-max control problem, with negatively weighted disturbances as in $\mathcal{H}_\infty$ control, that provides the tightest possible lower performance bound on the original problem of interest and whose value function is easily computed. The new method is demonstrated via a numerical example for a system with box constrained~input. 
\end{abstract}

\section{Introduction}

Consider the discrete-time linear time-invariant system
\begin{equation} \label{dynamics}
\begin{aligned}
x_{t+1} &= Ax_t + Bu_t + Gw_t \\
    z_t &= Cx_t + Du_t
\end{aligned}
\end{equation}
with state $x_t\in \mathbf{R}^n$, input $u_t\in\mathbf{R}^m$,  unknown exogenous disturbance $w\in \mathbf{R}^l$, and costed/controlled output $z_t \in \mathbf{R}^p$ at each time $t = \{0,1,\dots\}$.   We assume throughout that the initial state $x_0$ is known and that the state is directly observable at each time step.  
We further assume that the inputs are subject to a compact constraint $u_t \in \Uset$, that the disturbances are constrained to a compact set $w_t \in \Wset$, that the pairs $(A,B)$ and $(A,C)$ are controllable and observable, respectively, and that $D$ is full rank.   We denote by $\Wsetinf$ the set of all infinite-horizon disturbance sequences generated by~$\Wset$.


We define a stationary control policy $\pi: \mathbf{R}^n \rightarrow \Uset$ for the system \eqref{dynamics} as a decision rule mapping observed states to control inputs, i.e.\ $u_t = \pi(x_t)$.   We denote the set of all such stationary policies satisfying the system constraints for all infinite-horizon disturbance sequences by $\Pi$.   The linear min-max control problem is to design a control policy $\pi \in \Pi$ that minimizes the worse-case value of some objective function over all admissible disturbance sequences, starting from some known initial state.  

Given an initial state $x := x_0$, a control policy $\pi$ and a disturbance sequence $\mathbf{w}\in\Wsetinf$, we consider a cost function in the form
\begin{equation} \label{cost}
J(x_0,\pi,\mathbf{w}) = \sum_{t=0}^\infty \alpha^t \ell(z_t,w_t)
\end{equation}
where $\ell: \mathbf{R}^n \times \Uset \times \Wset \rightarrow \mathbf{R}$ is a cost function and $\alpha \in (0,1)$ is a discount factor. 
%
%
The optimal value function and policy are given by
\begin{equation} \label{optcost}
\begin{aligned}
V(x_0) &= \min_{\pi \in \Pi} \max_{w \in \Wsetinf} \sum_{t=0}^\infty \alpha^t \ell(z_t,w_t) \\
     \pi^*(x) &= \arg \min_{\pi \in \Pi} \max_{w \in \Wsetinf} \sum_{t=0}^\infty \alpha^t \ell(z_t,w_t)
\end{aligned}
\end{equation}

It is well known that, given the preceding assumptions and certain mild conditions on the stage cost $\ell$, a stationary optimal control policy $\pi^*$ exists \cite[Ch. 5]{BertsekasShreveBook}.
However, it is extremely difficult to compute the optimal cost and policy in general, because the optimization problem to be solved is over an infinite dimensional function space and over the worst case disturbance sequence. 

It was shown in \cite{mayne2006characterization} that for constrained finite-horizon problems with a cost convex and quadratic in $z$ and concave and quadratic in $w$, the optimal value function is piecewise quadratic and the optimal control policy piecewise affine over a polytopic partition of the state space, though computing and storing the policy is possible only for very small systems. Recent research has focused on developing tractable but suboptimal policies for systems in which the approach of \cite{mayne2006characterization} is intractable. Examples include various formulations of min-max model predictive control \cite{lall1994game,kothare1996robust,lofberg2003approximations,goulart2009control}. This immediately raises questions about the degree of  suboptimality of such policies.  To the best of our knowledge there is currently no general method to answer these questions for min-max problems. 

We therefore propose a method for computing a lower bound on the optimal cost for input constrained problems, which provides a way to bound the degree of suboptimality of any control policy constructed by any given design method. Our approach is motivated by recent work from Wang and Boyd \cite{wang2009performance}, who propose a similar approach to computing performance bounds for stochastic optimal control problems  (i.e.\ those with an expected value, rather than a min-max, cost). The central idea is to find an unconstrained linear min-max control problem for which the optimal cost can be efficiently computed that best approximates from below the optimal cost of the original constrained problem. 

In particular, we construct an auxiliary convex-concave quadratic cost function for the system \eqref{dynamics} without constraints, and exploit this auxiliary function to compute a lower bound on the value function $V$ of the constrained system.  The particular choice of a convex-concave auxiliary cost function is motivated by the fact that such cost functions appear naturally in $\mathcal{H}_\infty$ control design problems, for which exact performance value functions can be computed.  Our resulting bounds are not generic, but are instead computed individually for each problem instance. Such bounds can be useful in evaluating the performance of suboptimal policies in cases where the optimal policy cannot be efficiently computed: when the gap between the lower bound and the performance of a suboptimal policy is small, one can conclude that the suboptimal policy is nearly optimal. 

%
\begin{figure*}[t]
\newcounter{tempMathCtr}
\setcounter{tempMathCtr}{\value{equation}}
\setcounter{equation}{\value{hugeEquationNumber} - 1}
\begin{equation} \label{relaxation4}
\small
\begin{aligned}
& \text{minimize} \quad \mathbf{Tr}(P) \\
& \text{subject to} \quad P \succ 0, \quad X \succ 0 \\ 
& \left[\begin{array}{cccc}N\t(AXA\t-X) & N\t AX & N\t (AXA\t-X)M & N\t G \\XA\t N & \gamma Q\inv & XA\t M & 0 \\M\t(AXA\t-X)N & M\t AX &  M\t(AXA\t - X -\gamma B R^{-1}B\t)M & M\t G\\
G\t N & 0 & G\t M&-\gamma I \end{array}\right]\prec 0 \\ & -\gamma I + G\t P G \succ 0,\quad  \left[\begin{array}{cc}X & I \\I & P\end{array}\right] \succeq 0.
\end{aligned}
\tag{$\mathcal{SDP}$}
\end{equation}
\hrulefill
\vspace*{4pt}
\setcounter{equation}{\value{tempMathCtr}}
\end{figure*}
%
%

The paper is organized as follows. Section \ref{sec:II} briefly reviews solution methods for unconstrained min-max problems, which are closely related to the $\Hinf$ control problem. Section \ref{sec:III} describes the method for computing the proposed bound. Section \ref{sec:IV} provides a numerical example for which the bound is computed. Finally, Section \ref{sec:V} gives concluding remarks.

%

\section{\!\!Unconstrained Problems with Quadratic Cost\!}\label{sec:II}
%
%
%
%
This section reviews solution methods for unconstrained min-max problems whose objective functions are convex-concave in $(z,w)$, which will be used to compute lower bounds for input constrained problems in Section \ref{sec:III}. If the stage cost for some unconstrained problem is less than or equal to the stage cost for the constrained problem, then we will show that the unconstrained optimal cost is a lower bound on the constrained optimal cost for a certain set of initial states. One can then optimize the parameters in the stage cost of the unconstrained problem to find tighter lower bounds for the original constrained problem. 

We begin by recalling that a special case, closely related to the $\Hinf$ optimal control problem, in which the min-max optimal policy can be efficiently computed is when the stage cost is convex and quadratic in $(x,u)$ and concave and quadratic in $w$, with discount factor $\alpha = 1$, i.e.\
\begin{equation}
 \label{stage}
\bar{J} = \sum_{t=0}^\infty \|Cx_t + Du_t\|^2 - \gamma^2 \|w_t\|^2, \quad \gamma >0
\end{equation}
and there are no constraints, i.e.\ when $\Uset=\mathbf{R}^m$ and $\Wset=\mathbf{R}^l$. In this case, the optimal value function is quadratic in the initial state, the optimal state feedback policy is linear, and the closed-loop system has $\ell_2$ gain bounded by $\gamma$, whenever $\gamma$ is larger than the $\Hinf$ optimal gain. Furthermore, the optimal cost and feedback gain can be computed efficiently from the problem data using a variety of well-known techniques, see e.g.\ \cite{de1992discrete,gahinet1994linear,basar2004h}. 

One standard solution method for this problem is based on dynamic programming recursion using the Isaacs equation, which is the discrete-time dynamic game counterpart to the Bellman equation from dynamic programming \cite{basar2004h}. Suppose that $(P,\bar P) \in \mathbf{R}^{n\times n}$  are symmetric matrices that satisfy the generalized algebraic Riccati equation
\begin{subequations}\label{eqn:dynProg}
\begin{align}
P&= Q + A^T\bar PA - A^T\bar PB(R + B^T\bar PB)^{-1}B^T\bar PA \label{eqn:dynProg:P}\\
\bar P&\eqdef  P + PG(\gamma^2I - G^T PG)^{-1}G^T P, \label{eqn:dynProg:Pbar}
\end{align}
\end{subequations}
where we have assumed for simplicity that $Q = C\t C > 0$, $R = D\t D > 0$, $C\t D = 0$ and $D\t C = 0$.
If $P\succ 0$ is a solution to \eqref{eqn:dynProg}, then $V(x_0)= x_0\t P x_0$ is the optimal value function, and the optimal controller and disturbance policies are $u = Kx$ and $w = K_wx$, respectively, where 
\begin{equation}\label{eqn:Kdef}
\begin{aligned}
K &= -(R + B\t\bar PB)^{-1}B\t \bar PA \\
K_w &=(\gamma^2I - G\t P G)^{-1}G\t P (A + BK). 
\end{aligned}
\end{equation}
For $\alpha \in (0,1)$ the optimal cost and policies have the same structure but with some problem data scaled by the discount factor: $A \rightarrow \sqrt{\alpha}A$, $B \rightarrow \sqrt{\alpha}B$ and $\gamma \rightarrow \gamma/\sqrt{\alpha}$ (see the Appendix). In other words, the optimal cost and controllers for the discounted problem can be computed by solving an associated undiscounted problem. 

Alternative solution methods involve relaxing the Riccati equation \eqref{eqn:dynProg:P} to an inequality and solving a semidefinite programming problem.  If \eqref{eqn:dynProg:P} is relaxed to the right (i.e\ by replacing $=$ with $\succeq$), it can be shown that the resulting inequality is equivalent to the linear matrix inequality obtained by applying the Bounded Real Lemma to the system \eqref{dynamics} in close-loop with the optimal controllers $u_t = Kx_t$ and $w_t = K_wx_t$ \cite{gahinet1994linear,de1992discrete} (see  Appendix). The optimal cost and state feedback gain matrices are given by the solution to the following optimization problem in variables $P$ and $K$:
\begin{equation} \label{relaxation2}
\begin{aligned}
& \text{minimize} \quad  \mathbf{Tr}(P) \\
& \text{subject to} \quad P \succ 0 \\
& \left[\begin{array}{cccc}-P\inv & A+BK & G & 0 \\(A+BK)\t & -P & 0 & (C+DK)\t \\G\t & 0 & -\gamma I & 0 \\0 & C+DK & 0 & -\gamma I \end{array}\right]\prec 0. 
\end{aligned}
\end{equation}
The usual approach to solving \eqref{relaxation2} is to transform this problem into a semidefinite program in variables $P\inv$ and $KP\inv$.  However, in the next section we will allow $D$ to vary in order to obtain the tightest possible lower bound for the problem \eqref{optcost}, so an alternative to \eqref{relaxation2} is preferred.  As shown in \cite{gahinet1994linear}, the optimal cost can also be computed separately from the optimal state feedback via the semidefinite program \eqref{relaxation4} in variables $P$ and $X$.
%
%
%
%
%
%
%
The matrices $N$ and $M$ in \eqref{relaxation4} are bases for the null space of $B\t$ and its orthogonal complement, respectively. Note that in contrast to \eqref{relaxation2},  the problem \eqref{relaxation4} is affine in $R\inv$.  This property will be useful when computing lower bounds for the original problem  \eqref{optcost}.

\begin{remark}
The problem \eqref{relaxation4} is in contrast with the computation of performance bounds for constrained linear stochastic control problems in \cite{wang2009performance}, in which bounds were obtained by relaxing the standard linear quadratic regulator Riccati equation in the other direction and maximizing over the resulting linear matrix inequality. If \eqref{eqn:dynProg:P} is relaxed to the left (i.e.\ by replacing $=$ with $\preceq$), one obtains the following optimization problem in variables $P$ and $\bar P$
\begin{equation} \label{relaxation1}
\begin{aligned}
& \text{maximize} && \mathbf{Tr}(P) \\
& \text{subject to} && \left[\begin{array}{cc}R+B^T\bar P B & B^T\bar PA \\A^T\bar P B & Q+A^T\bar P A- P \end{array}\right] \succeq 0 \\ & && \bar P = P + PG(\gamma^2I - G^T PG)^{-1}G^T P \\ & && P \succeq 0.
\end{aligned}
\end{equation}
Unfortunately, this problem is not convex. For stochastic linear quadratic regulator problems considered in \cite{wang2009performance} (recovered with $\gamma \rightarrow \infty$ and thus $\bar P \rightarrow P$), the situation is easier since the problem becomes a semidefinite program in $P$ and is also concave in $(Q,R)$.
\end{remark}

\section{Performance Bound for Problems with Input Constraints}\label{sec:III}
This section describes a method to obtain lower bounds on performance for min-max problems with input constraints. The central idea is to find an unconstrained problem	 whose optimal cost, which as shown in the previous section can be efficiently computed, is a lower bound on the optimal cost of the constrained problem. To this end, we associate $Q\succ 0$, $R\succ 0$, $\gamma>0$ with the stage cost of some unconstrained problem and denote by $P$, $K$ and $K_w$ the corresponding optimal cost and gain matrices given by \eqref{eqn:dynProg}--\eqref{eqn:Kdef}. The closed-loop system for this unconstrained problem is $x^+ = A_{cl} x$ where $A_{cl} = A+BK+GK_w$. We also define the set 
\begin{equation}\label{eqn:nastyX}
\!X \!=\! \{ x\in \mathbf{R}^n | K_w (A+BK+GK_w)^k x \in \Wset, \, \forall k\in \mathbb{N} \},\!\!\!
\end{equation}
which corresponds to the set of states for which the optimal unconstrained disturbance sequence always remains in $\Wset$. We have the following result, which gives a basic lower performance bound for initial states in $X$:
\begin{lemma}
Suppose $Q$, $R$, $\gamma$ and $s \in \mathbf{R}$ satisfy
\begin{equation} \label{semiinf}
\begin{aligned}
x^T Q x + u^T R u - \gamma^2 w^T w + s  \leq \ell(z,w) ,\\  \forall x\in \mathbf{R}^n,\ \forall u\in \Uset, \ \forall w\in \Wset.
\end{aligned}
\end{equation}
Then we have for $\alpha \in (0,1)$
\begin{equation}\label{eqn:infSumBound}
x_0\t P x_0 +\frac{s}{1-\alpha} \leq V(x_0), \quad \forall x_0 \in X
\end{equation}
where $P$ is the optimal cost matrix for the unconstrained min-max problem associated with $Q$, $R$, and $\gamma$.
\end{lemma}
\begin{proof} Observe that \eqref{semiinf} implies 
\begin{equation} \label{costineq}
\begin{aligned}
\min_{\pi \in \Pi} &\max_{w\in \Wsetinf} \sum_{t=0}^{\infty}\alpha^t [x_t^T Q x_t + u_t^T R u_t - \gamma^2 w_t^T w_t +s] \\ &\leq \min_{\pi \in \Pi} \max_{w\in \Wsetinf}  \sum_{t=0}^{\infty}\alpha^t l(z,w).
\end{aligned}
\end{equation}
The inequality \eqref{costineq}  remains valid if the input constraint in the set of admissible policies on the outer left-hand-side minimization is dropped. If a given $x_0$ is in the set $X$ then the constraint in the inner left-hand-side maximization over disturbance sequences can be dropped without changing the value because the constraint is never active by the definition of $X$. The term involving $s$ in the left-hand-side of \eqref{costineq} can be taken out of the optimizations and evaluated to yield the second term on the left-hand-size of \eqref{eqn:infSumBound}. The rest of the left-hand side is then the unconstrained optimal min-max cost and the right-hand-side is the optimal cost of the constrained problem, and the result follows immediately. 
\end{proof}

\subsection{Optimizing the bound}
Lemma 1 gives a basic but efficiently computable lower bound on the value function $V$ for our constrained minimax problem. We next consider the problem of maximizing this lower bound by allowing the parameter $R\inv$ in \eqref{relaxation4} to vary\footnote{One could also contemplate allowing $Q\inv$ to vary. Here, since we do not consider state constraints, \eqref{semiinf} must hold on all of $\mathbf{R}^n$, and one does not gain anything by allowing $Q\inv$ to vary. Therefore, as in \cite{wang2009performance}, we fix $Q$ to the cost for the constrained problem.  }. Let $J^*(R\inv)$ denote the optimal value of \eqref{relaxation4}. To optimize the lower bound, we can solve the following optimization problem over the parameters $R\inv$ and $s$ 
\begin{equation} \label{opt1}
\begin{aligned}
& \text{maximize} && J^*(R\inv) + \frac{s}{1-\alpha} \\
& \text{subject to} && \eqref{semiinf}, \quad R\inv \succ 0.
\end{aligned}
\end{equation}
This problem presents two difficulties. First, since \eqref{relaxation4} is a minimization problem and \eqref{opt1} is a maximization problem, the optimization cannot be done jointly over $P$, $X$, $R\inv$ and~$s$. We can however form the dual of \eqref{relaxation4} and maximize jointly over $R\inv$, $s$ and the dual variables. This results in the following bilinear problem with variables $R\inv$, $s$ and dual variables $Z$, $\Phi$ and $\Lambda$: 
\begin{equation} \label{dual2}
\begin{aligned}
& \text{max} && \hspace{-5ex}-\gamma \mathbf{Tr}(Z) + \mathbf{Tr}(F(R\inv)\Lambda) + 2\mathbf{Tr}(\Phi_{12}) + \frac{s}{1-\alpha} \\
& \text{subject to} && \Phi_{22} + I + G\t Z G = 0,\quad \eqref{semiinf},\quad R\inv \succ 0 \\ 
& && \Phi_{11}+A\t N \Lambda_{11} N\t A - N \Lambda_{11} N\t + 2A\t N \Lambda_{12} \\
& && \quad + 2A\t N \Lambda_{13} M\t A - 2N \Lambda_{13} M\t  + \Lambda_{22} \\
& && \quad + 2\Lambda_{23} M\t A + A\t M \Lambda_{33} M\t A \\
& && \quad - M\Lambda_{33} M\t =0 \\
& && Z \succeq 0, \quad \Phi = \left[\begin{array}{cc}\footnotesize\Phi_{11} & \Phi_{12} \\\Phi_{12}^T & \Phi_{22}\end{array}\right] \succeq 0 \\
& && \Lambda = \footnotesize\left[\begin{array}{cccc}\Lambda_{11} & \Lambda_{12} & \Lambda_{13} & \Lambda_{14} \\\Lambda_{12}^T & \Lambda_{22} & \Lambda_{23} & \Lambda_{24} \\\Lambda_{13}^T & \Lambda_{23}^T & \Lambda_{33} & \Lambda_{34} \\\Lambda_{14}\t & \Lambda_{24}\t & \Lambda_{34}\t & \Lambda_{44}\end{array}\right] \succeq 0.
\end{aligned}
\end{equation}
where $F(R\inv):=$
\begin{equation}  \footnotesize
\begin{aligned} &\left[\begin{array}{cccc}0 & 0 & 0 & N\t G \\0 & -\gamma Q\inv & 0 & 0 \\0 & 0 & -\gamma M\t B R\inv B\t M & M\t G \\G\t N & 0 & G\t M & -\gamma I\end{array}\right].
\end{aligned}\normalsize
\end{equation}
Observe that there is a product term involving $\Lambda_{33}$ and $R\inv$ in the dual objective, so this problem is bilinear, and consequently non-convex. However, local and global methods for solving such bilinear optimization problems (e.g. PENBMI \cite{kocvara2005penbmi}) are available. In the next section we adopt the simplest approach,  of alternately solving semidefinite programs by fixing one variable in the product and optimizing over the other. We show via numerical example that this method can be used to produce improved lower bounds.

The second difficulty is that \eqref{semiinf} is a semi-infinite constraint in general since it must hold for all points in $\mathbf{R}^n$, $\Uset$ and $\Wset$. As shown in \cite{wang2009performance}, we recall in the following subsections that when $\ell(z,w)$ is quadratic the constraint \eqref{semiinf} can be enforced exactly when $\Uset$ is a finite set, and that it can be replaced with a conservative approximation via the S-procedure in other cases.

\begin{remark}
The reason why the dual of \eqref{relaxation4} is preferred to the dual of \eqref{relaxation2} is that the latter has primal variables multiplying $D$ which we allow to vary in the dual problem. As a consequence, the associated dual problem has bilinear \emph{equality} constraints. Since $R\inv$ enters as a constant term in \eqref{relaxation4} , the dual problem has the bilinearity in the objective, which we find more convenient to work with.
\end{remark}

\subsection{Verifying valid initial states}
The bound \eqref{eqn:infSumBound} is valid only for initial states $x_0 \in X$, which is characterized by an infinite set of constraints in \eqref{eqn:nastyX}. We now describe a procedure to verify that a given $x_0$ is indeed contained in $X$. Suppose we have a symmetric matrix $S$ and scalar $\beta \in \mathbf{R}$ such that 
\mbox{$\bar \Wset = \{ w\in \mathcal{R}^l | w\t S w \leq \beta \} \subseteq \Wset$}. 
Suppose also that $x\t H x $ is a dissipated quantity for the optimal unconstrained closed-loop system $x^+ = A_{cl}x$, where $A_{cl} = A+BK+GK_w$ (i.e. $H$ satisfies $A_{cl}\t H A_{cl} - H \preceq 0$). To show $x_0 \in X$, it is sufficient to show
$$x\t H x \leq x_0 H x_0 \Rightarrow x\t K_w\t S K_w x \leq \beta. $$ 
By the S-procedure (see e.g. \cite{boyd1994linear}) this is equivalent to the existence of $\lambda \geq 0$ satisfying
$$K_w\t S K_w \preceq \lambda H, \quad -\beta \leq -\lambda x_0\t H x_0.$$
Thus, to verify $x_0 \in X$ it is sufficient to show that there exists $H$ and $1/\lambda$ satisfying the LMI conditions
\begin{equation} \label{verify}
\begin{aligned}
(1/\lambda)K_w\t S K_w \preceq H, \quad -\beta/\lambda \leq -x_0\t H x_0, \\ \quad A_{cl}\t H A_{cl} - H \preceq 0, \quad 1/\lambda > 0, \quad H \succeq 0.
\end{aligned}
\end{equation}
If a feasible point is found, the bound is valid for all initial states in the set $\{x\in \mathbf{R}^n | x\t H x \leq x_0\t H x_0 \}$. For a particular initial state, this procedure may be somewhat conservative. The conservatism can be reduced by evaluating a finite number of the constraints in $X$ and then solving the feasibility problem from the last state evaluated.

\subsection{Finite input sets}
Suppose that $\ell(z,w)$ is a quadratic function in the form 
$$\ell(z,w) = x\t Q_0 x + u\t R_0 u - \gamma_0^2 w\t w$$
(with $\gamma_0 = 0$ possible) and the input constraint set is finite ($\Uset = \{u_1,...,u_k\}$). If we fix $Q = Q_0$ and $\gamma \geq \gamma_0$, then \eqref{semiinf} reduces to 
\begin{equation}
u_i\t R u_i + \frac{s}{1-\alpha} \leq u_i\t R_0 u_i, \quad i = 1,...,k,
\end{equation}
which can be expressed using Schur complements as LMIs in $R\inv$ and $s$:
\begin{equation} \label{finite}
\begin{bmatrix}
\left(u_i\t R_0 u_i - \frac{s}{1-\alpha}\right) & u_i\t \\
u_i & R\inv
\end{bmatrix}
\succeq 0, \quad i = 1,...,k.
\end{equation}
Replacing \eqref{semiinf} in \eqref{dual2} with \eqref{finite} and solving the bilinear problem can yield an improved lower bound.

\subsection{S-Procedure relaxation}
Suppose again the stage costs are quadratic and now that we have $R_1,...,R_M$ and $s_1,...,s_M$ for which
\begin{equation} \label{quadineq}
\Uset  \subseteq \bar U = \{ u | u\t R_i u + s_i \leq 0, \ i = 1,...,M \}.
\end{equation}
Setting $Q = Q_0$ and $\gamma \geq \gamma_0$ again, a sufficient condition via the S-procedure (see \cite{boyd1994linear} and \cite{wang2009performance}) for \eqref{semiinf} to hold is the existence of nonnegative $\lambda_1,...,\lambda_M$ satisfying
\begin{equation}
R-R_0 \preceq \sum_{i=1}^M \lambda_i R_i, \quad \frac{s}{1-\alpha} \leq \sum_{i=1}^M \lambda_i s_i.
\end{equation}
The first inequality can be written using Schur complements as an LMI in $R\inv$
\begin{equation} \label{sproc}
\left[\begin{array}{cc}R_0 + \sum_{i=1}^M \lambda_i R_i & 0 \\0 & R\inv\end{array}\right] \succeq 0.
\end{equation}
Again, replacing \eqref{semiinf} in \eqref{dual2} with \eqref{sproc} and solving the bilinear problem can yield an improved lower bound.

\section{Numerical Example}\label{sec:IV}
In this section we compute a performance bound for an example system in which the input constraint set $\Uset$ is a box. The disturbance constraint set is taken to be a unit ball, and the cost is taken to be quadratic: $\ell(z,w) = x\t Q_0 x + u\t R_0 u - \gamma_0 w\t w$. System and cost matrices with dimensions $n=l=4$, $m=2$, and $p = 6$ were randomly generated, with entries from $A$, $B$, $G$, and Cholesky factors of $Q_0$ and $R_0$ drawn from a standard normal distribution. The disturbance weight $\gamma_0$ was set to ten percent larger than the unconstrained $\Hinf$ optimal value, and the discount factor was set to $\alpha = 0.95$. 

For the associated unconstrained problem, we set $Q = Q_0$ and $\gamma = \gamma_0$. The problem \eqref{relaxation4} is first solved with $R = R_0$ to obtain a basic lower bound. Then the bound is optimized by solving \eqref{dual2} by alternately fixing $R\inv$ and $\Lambda_{33}$. The corresponding semidefinite programs were solved using the modeling language YALMIP \cite{lofberg2004yalmip} and the cone solver SeDuMi \cite{sturm1999using}.

A box constraint is defined by $$\Uset = \{ u\in \mathcal{R}^m |\  \norm{u}_\infty \leq U_{max}\}.$$ This constraint can be represented by a set of quadratic inequalities as follows
$$u\t e_i e_i\t u  \leq U_{max}^2, \quad i=1,...,m, $$ where $e_i$ is the $i$th unit vector. In relation to \eqref{quadineq}, we have $R_i = e_i e_i\t$ and $s_i = -U_{max}^2$. 

We now describe a specific example. The system matrices are: \footnotesize
$$A = \left[\begin{array}{cccc}0.434 & 0.050 & 0.212 & 0.007 \\0.264 & 0.001 & 0.092 & 0.419 \\0.307 & 0.255 & 0.371 & 0.359 \\0.364 & 0.003 & 0.291 & 0.427\end{array}\right], \ B = \left[\begin{array}{cc}0.739 & 0.550 \\0.371 & 0.748 \\0.323 & 0.760 \\0.491 & 0.472\end{array}\right] $$
$$ G = \left[\begin{array}{cccc}0.802 & 0.666 & 0.737 & 0.629 \\0.471 & 0.677 & 0.866 & 0.793 \\0.203 & 0.9425 & 0.991 & 0.449 \\0.576 & 0.7701 & 0.504 & 0.524\end{array}\right], \ R_0 = \left[\begin{array}{cc}0.262 & 0.560 \\ * & 1.33\end{array}\right] $$
$$Q_0 = \left[\begin{array}{cccc}0.105 & 0.286 & 0.221 & 0.271 \\ * & 0.929 & 0.618 & 0.687 \\ * & * & 1.22 & 0.854 \\ * & * & * & 0.873\end{array}\right]. $$ \normalsize
Solving \eqref{relaxation4} with $Q = Q_0$, $R = R_0$ and $\gamma = \gamma_0$ we obtain an optimal value of 3.526. We then run an alternating semidefinite programming procedure on \eqref{dual2} and obtain the value 6.42, which is $82\%$ improvement on the basic bound obtained from \eqref{relaxation4}. Sets of initial states for which these bounds are valid can be computed via~\eqref{verify}.

\section{Conclusions}\label{sec:V}
This paper developed a method to compute performance bounds for infinite-horizon linear min-max control problems with input constraints. The method requires a bilinear matrix inequality to be solved. We showed that improved bounds can be computed by solving semidefinite programs in an alternating fashion. Further research directions include computing bounds for problems with state constraints and comparison of the bound with suboptimal policies.

\vspace{-1ex}
\appendix{
\vspace{-1ex}
\subsection{Infinite-horizon discounted $\Hinf$ problem}
\vspace{-1ex}

This section derives the solution of the infinite-horizon unconstrained linear min-max optimal control problem with a discounting factor. The quadratic form $z\t Q z$ is denoted by $\norm{z}_Q^2$.
Given the system
\[
x^+ = Ax + Bu + Gw,
\]
consider the problem of minimizing the worst-case discounted cost function with discount factor $\alpha \in (0,1)$
\[
J(x_0,\pi,\mathbf{w}) \eqdef \sum_{k=0}^\infty \alpha^k (\norm{x}_Q^2 + \norm{u}^2_R - \gamma^2\norm{w}^2 ).
\]
The Isaacs equation is
\[
\begin{aligned}
V &= \min_u\max_w
\left[\ell(x,u,w) + \alpha V(Ax+Bu+Gw)\right], \\
&= \min_u
\left[\norm{x}^2_Q + \norm{u}_R^2 + J(x,u)\right]
\end{aligned}
\]
where
\[
J(x,u) \eqdef \max_w\left[ -(\gamma^2)\norm{w}^2 + \alpha V(Ax+Bu+Gw)\right].
\]
Assume that $V(x) =  x\t Px$.  The inner maximization problem can be solved explicitly by rewriting it first as
\[
\begin{aligned}
&J(x,u) \\
&= \max_w\left[ -\gamma^2\norm{w}^2 +\alpha \norm{Ax+Bu+Gw}_P^2  \right] \\
&=\max_w\left[-\norm{w}^2_{(\gamma^2I-G\t(\alpha P)G)} + w\t G\t(\alpha P)(Ax+Bu)\right] \\  
&\quad + \norm{Ax+Bu}_{\alpha P}^2 
\end{aligned}
\]
The optimizer $w^*(x,u)$ for the above is
\[
w^*(x,u) = \left[\gamma^2I-G\t(\alpha P)G\right]^{-1}G\t(\alpha P)(Ax+Bu),
\]
so that
\begin{align*}
J(x,u) &= (Ax+Bu) \bar P(Ax+Bu)\\
\bar P &\eqdef (\alpha P) + (\alpha P)G\left(\gamma^2I-G\t(\alpha P)G\right)^{-1}G\t (\alpha P).
\end{align*}
The value function can then be written as 
\[
\begin{aligned}
V(x) &= \min_u\left[\norm{x}^2_Q + \norm{u}^2_Q + J(x,u)\right]\\
       &= \min_u\left[\norm{x}^2_Q + \norm{u}^2_Q + (Ax+Bu)\t \bar P(Ax+Bu)\right].
\end{aligned}
\]
The rest of the argument then follows along the lines of those in \cite{basar2004h}, yielding
\begin{align}
\!\!P &= Q + A\t \bar P A - A\t\bar PB(R+B\t\bar PB)^{-1}B\t \bar P A\label{eqn:discP}\\
\!\!\bar P &\eqdef (\alpha P) + (\alpha P)G\left(\gamma^2I-G\t(\alpha P)G\right)^{-1}G\t (\alpha P).\label{eqn:discPbar}
\end{align}
The optimal input and disturbance policies are
\[
\begin{aligned}
K_u &= -(R+B\t \bar P B)^{-1}B\t \bar PA, \\
K_w &= (\gamma^2I-G\t(\alpha P)G)^{-1}G\t(\alpha P)(A+BK_u).
\end{aligned}
\]

We can now try to equate a solution of the above to some other undiscounted problem.    Rewrite \eqref{eqn:discPbar} as
\[
\begin{aligned}
\bar P &= (\alpha P) + \frac{1}{\alpha}(\alpha P)G\left(\frac{\gamma^2}{\alpha}I-G\t PG\right)^{-1}G\t (\alpha P)\\
&= \alpha \cdot \left[P + PG\left(\tilde{\gamma}^2I-G\t PG\right)^{-1}G\t P\right]
= \alpha \cdot \bar{\mcf{P}}
\end{aligned}
\]
where $\tilde\gamma \eqdef \gamma/\sqrt{\alpha}$ and 
\[
\bar{\mcf{P}} \eqdef P + PG\left(\tilde{\gamma}^2I-G\t PG\right)^{-1}G\t P.
\]
Then substitute $\bar P = \alpha\bar{\mcf{P}}$ into \eqref{eqn:discP} to obtain
\[
\begin{aligned}
\!P \!&=\! Q \!+\! A\t (\alpha\bar{\mcf{P}}) A \!-\!\! A\t(\alpha\bar{\mcf{P}})B(R\!+\!B\t(\alpha\bar{\mcf{P}})B)^{\!-1\!}B\t (\alpha\bar{\mcf{P}}) A \\
&= Q + \tilde A\t \bar{\mcf{P}} \tilde A - \tilde A\t\bar{\mcf{P}}\tilde B(R+\tilde B\t \bar{\mcf{P}}\tilde B)^{-1}\tilde B\t \bar{\mcf{P}} \tilde A
\end{aligned}
\]
where $\tilde A \eqdef \sqrt{\alpha}A$ and $\tilde B \eqdef \sqrt{\alpha}B$.
The conclusion is that the solution to the discounted $\Hinf$ problem can be obtained by solving the equations
\[
\begin{aligned}
\bar{\mcf{P}} &= P + PG\left(\tilde{\gamma}^2I-G\t PG\right)^{-1}G\t P\\
P &= Q + \tilde A\t \bar{\mcf{P}} \tilde A - \tilde A\t\bar{\mcf{P}}\tilde B(R+\tilde B\t \bar{\mcf{P}}\tilde B)^{-1}\tilde B\t \bar{\mcf{P}} \tilde A
\end{aligned}
\]
i.e.\ by solving an undiscounted problem using the problem data $\tilde A \leftarrow \sqrt{\alpha}A$, $\tilde B \leftarrow \sqrt{\alpha}B$ and $\tilde\gamma \leftarrow \gamma/\sqrt{\alpha}$.  

The value function for the discounted problem is then
$
V(x) = x\t Px,
$
where $P$ comes from the solution to the undiscounted $\Hinf$ problem using the modified problem data.
\begin{remark}
We note that the control that minimizes the infinite-horizon discounted cost is not necessarily stabilizing for the system \eqref{dynamics}, even if the optimal cost is finite. However, only the cost is used when determining a lower bound.
\end{remark}

\vspace{-1ex}
\subsection{Equivalence of Riccati Inequality and Bounded-Real LMI}
\vspace{-1ex}

This section demonstrates the equivalence between the Riccati inequality obtained by replacing $=$ with $\succeq$ in \eqref{eqn:dynProg} and the bounded-real LMI in \eqref{relaxation2}. We first recall a result from \cite{de1992discrete}.
\begin{thm}[{Thm. 2.2,\ \cite{de1992discrete}}]\label{thm:deSouza}
For the system
\begin{equation}\label{eqn:systemNoInput}
\begin{aligned}
x^+ &= Ax + Gw \\
z   &= Cx,
\end{aligned}
\end{equation}
the following statements are equivalent:
\begin{enumerate}
\item $A$ is a stable matrix and $\|C(zI-A)^{-1}G\|_\infty \le \gamma$.
\item There exists a stabilizing solution $P = P\t \succeq 0$ to the Riccati equation
\end{enumerate}
\begin{equation}\label{eqn:Riccati}
\!P\!=\!A\t \!PA \!+\! \frac{1}{\gamma^{\smash{2}}}A\t \!PG (I \!-\! \frac{1}{\gamma^{\smash{2}}}G\t\! P G )^{-1}G\t PA + C\t \!C.\! \hspace{-1.25ex}
\end{equation}
\end{thm}
\begin{claim}
If one substitutes $A \rightarrow (A+BK)$ and $C \rightarrow (C+DK)$ in \eqref{eqn:Riccati} with $K$ as in \eqref{eqn:Kdef}, then 
\eqref{eqn:Riccati}~$\Leftrightarrow$~\eqref{eqn:dynProg}.
\end{claim}
\begin{proof}
We have from \eqref{eqn:Riccati}  
\begin{align*}
P &= A\t PA  + A\t PG\left[\gamma^2 - G\t P G\right]^{-1}G\t PA + C\t C \\
&= A\t PA  + A\t (\bar P - P)A + C\t C = A\t \bar PA  + C\t C,
\end{align*}
where the first step uses the definition of $\bar P$ from \eqref{eqn:dynProg}.

Next replace $A \rightarrow (A+BK)$ and $C \rightarrow (C+DK)$ to get
\begin{align}
P &= (A+BK)\t \bar P(A+BK) + (C+DK)\t (C+DK) \notag\\
  &= (A+BK)\t \bar P(A+BK)  + Q + K\t R K \notag\\
  &= A\t \bar P A + Q + K\t (R + B\t \bar P B)K \notag \\ & \qquad + A\t \bar PBK + K\t B\t \bar P A\label{eqn:expandedK},
\end{align}
where the first step comes from the assumptions $D\t D = R$, $C\t C = Q$, and $C\t D = 0$.  Recalling the definition of $K$ in \eqref{eqn:Kdef}, the final three terms in \eqref{eqn:expandedK} can be rewritten as 
\begin{align}
K\t (R &+ B\t \bar P B)K + A\t \bar PBK + K\t B\t \bar P A \\ &= - A\t \bar PB(R + B\t \bar P B)^{-1}B\t \bar P A.
\end{align}
Substituting this into \eqref{eqn:expandedK} results in 
\[
P = A\t \bar P A + Q  - A\t \bar PB(R + B\t \bar P B)^{-1}B\t \bar P A,
\]
which is identical to \eqref{eqn:dynProg}.
\end{proof}
Finally, if we now relax \eqref{eqn:Riccati} by replacing  $=$ with $\succeq$, apply the Schur complement lemma three times and substitute $A \rightarrow (A+BK)$ and $C \rightarrow (C+DK)$, we arrive at the bounded-real LMI in \eqref{relaxation2}.

}

\vspace{-2ex}

{
\renewcommand{\baselinestretch}{0.91}
\bibliographystyle{unsrt}  
\bibliography{refspb}  
}

%
%
%
%
%
%
%
%

\end{document}